\documentclass[10pt,a4paper]{article}
\usepackage[utf8]{inputenc}
\usepackage{amsmath,amsthm,amsfonts,amsthm,amssymb,color}
\usepackage[margin=1in]{geometry}
\usepackage{todonotes}
\usepackage{tcolorbox}
\usepackage{tikz} \usetikzlibrary{matrix}
\usepackage{enumitem}

\renewcommand{\geq}{\geqslant}
\renewcommand{\leq}{\leqslant}

\renewcommand{\le}{\leqslant}

\def\cjref#1{Conjecture~$\ref{#1}$}

\DeclareMathOperator{\per}{per}
\DeclareMathOperator{\tr}{tr}

\theoremstyle{plain}
\newtheorem{theorem}{Theorem}[section]
\newtheorem{lemma}{Lemma}[section]

\theoremstyle{definition}
\newtheorem{definition}{Definition}
\newtheorem{conjecture}{Conjecture}

\newtheorem{problem}{Problem}

\numberwithin{equation}{section}

\newcommand{\blue}[1]{{\color{blue}#1}}

\begin{document}

\title{What did Ryser Conjecture?}
\author{Darcy Best and Ian M. Wanless\\
\small School of Mathematical Sciences\\[-0.5ex]
\small Monash University\\[-0.5ex]
\small Vic 3800, Australia\\
\small\tt \{darcy.best,ian.wanless\} \ @monash.edu
}
\date{}

\maketitle

\begin{abstract}
Two prominent conjectures by Herbert J.\ Ryser have been falsely
attributed to a somewhat obscure conference proceedings that he wrote
in German.  Here we provide a translation of that paper and try to
correct the historical record at least as far as what was conjectured
in it.

The two conjectures relate to transversals in Latin squares of odd
order and to the relationship between the covering number and the
matching number of multipartite hypergraphs.
\end{abstract}

\section*{Introduction}

The aim of this document is to make readily available the contents of the
following conference proceedings:

\medskip\noindent
H.\,J.~Ryser, Neuere Probleme in der Kombinatorik, {\it Vortrage
\"{u}ber Kombinatorik Oberwolfach}, 24--29 July (1967), {69--91}.

\medskip\noindent
hereafter referred to as [Rys67]. The reason for doing this is that 
[Rys67] has been attributed as the source of several 
important conjectures. However, it seems that these attributions are
not correct. So this is an attempt to straighten out the historical record.

We should begin with the important caveat that neither of us are
fluent in German, the language of the original document. However we
did consult native speakers Kevin Leckey and Anita Leibenau on key passages,
and rely heavily on google translate elsewhere. We could not have managed
without those people/resources, but they should not be blamed for any errors
in what is presented below.

The conjectures of present interest are:

\begin{conjecture}\label{cj:parity}
The number of transversals in a Latin square of order $n$ is
congruent to $n$ mod $2$.
\end{conjecture}

\begin{conjecture}\label{cj:odd}
Every Latin square of odd order has at least one transversal.
\end{conjecture}

\begin{conjecture}\label{cj:hyper}
In an $r$-partite hypergraph, $\tau\le(r-1)\nu$ where $\tau$ is the
covering number and $\nu$ is the matching number.
\end{conjecture}

All three have at times been attributed to [Rys67], but only
\cjref{cj:odd} is present in that document. The even half of
\cjref{cj:parity} was proved by Balasubramanian \cite{Bal90}, who
attributed the conjecture to [Rys67]. The odd half is false, and there
are many counterexamples for orders 7 and above, as has been noted in
numerous places including in \cite{CW05,MMW06,plexes}.  The suggestion
that Ryser made \cjref{cj:parity} has been repeated many times,
including in \cite{AA04,MMW06,plexes,surv1,surv2}.  It is possible
that Ryser did make this conjecture verbally, but we have been unable
to find any evidence that he made it in print. The one thing we are
certain of is that it does not appear in [Rys67]. If anyone can shed
any light on where Balasubramanian obtained the conjecture, or point
to an earlier reference to it, we would be very interested to hear
from them.

\cjref{cj:parity} does of course imply \cjref{cj:odd}, and
\cjref{cj:odd} is present in [Rys67], as we will see.  In papers on
transversals of Latin squares, \cjref{cj:odd} is often referred to as
Ryser's conjecture.

\bigskip

We now turn to \cjref{cj:hyper}, which is also often referred to as
Ryser's conjecture, and has seen a flurry of activity recently.
\cjref{cj:hyper} was attributed to [Rys67] in \cite{ABW16},
due to a misunderstanding between the authors. This mistake was
then copied in \cite{FHMW17}. Again we are unaware of any evidence
of Ryser making \cjref{cj:hyper} in print, although it does appear
in an equivalent form in the thesis of his student  
J.\,R.~Henderson \cite[p.26]{Hen71}. It does not appear in [Rys67].

\bigskip

The following pages give a translation of [Rys67].  We took the
liberty of correcting a couple of clear and simple typos, but have
otherwise tried to stay true to the original. Some commentary has been
added in footnotes in \blue{blue}.

  \let\oldthebibliography=\thebibliography
  \let\endoldthebibliography=\endthebibliography
  \renewenvironment{thebibliography}[1]{%
    \begin{oldthebibliography}{#1}%
      \setlength{\parskip}{0.4ex plus 0.1ex minus 0.1ex}%
      \setlength{\itemsep}{0.4ex plus 0.1ex minus 0.1ex}%
  }%
  {%
    \end{oldthebibliography}%
  }

\newpage

\centerline{\Large H. J. Ryser: \ Recent Problems In Combinatorics}


\maketitle
 

\section{Latin Squares}
 
Let $n$ be a natural number, and $S$ be a set of $n$ distinct elements, such as the digits $1,2,\dots,n$. A Latin rectangle of order $r \times s$ ($r$, $s$ natural numbers) is an $r \times s$ matrix $[a_{ij}]_{r \times s}$ of elements of $S$ which has no repetition in any row or column:
$$a_{i,j} = a_{i,k} \implies j=k \qquad (i = 1, \dots , r),$$
$$a_{i,j} = a_{k,j} \implies i=k \qquad (j = 1, \dots , s).$$

Thus, in every Latin rectangle $r \leq n$ and $s \leq n$. Since this definition does not depend on the order of the digits, nor on the distinction between columns and rows, one can restrict oneself to the consideration of Latin rectangles where $s = n$ and the first row is $1, 2, \dots, n$ in natural order:
$$a_{1,j} = j \qquad (j = 1, \dots , s).$$
These rectangles are called normalized.

If $r = s = n$, then we call it a Latin square of order $n$. The square is called normalized if both the first row and the first column contain the numbers $1, 2, \dots, n$ in natural order.
$$a_{1,j} = j \qquad (j = 1, \dots , n).$$
$$a_{i,1} = i \qquad (i = 1, \dots , n).$$

Some examples of small normalized Latin squares are
$$\left[\begin{array}{c}1\end{array}\right] \qquad
\left[\begin{array}{cc}1&2\\2&1\end{array}\right] \qquad
\left[\begin{array}{ccc}1&2&3\\2&3&1\\3&1&2\end{array}\right]$$

This definition is now followed by the combinatorial question about the number of different objects defined. Obviously, you get only one normalized $1 \times n$ Latin rectangle. The number of different $2 \times n$ Latin rectangles\blue{\footnote{\blue{For this statement to be true the intention must have been for the first row to be in order, but the second row to be unrestricted.}}} is equal to $D_n$, the number of permutations without any fixed points (``derangements''); this is easy to calculate: 
$$D_n = n! \cdot \left(1 - \frac{1}{1!} + \frac{1}{2!} - \cdots + (-1)^n \frac{1}{n!} \right) \approx \frac{n!}{e}.$$

However, an explicit expression for the number of $3 \times n$ Latin rectangles is quite complicated, and for $4\times n$ Latin rectangles, the solution to the problem is unknown.

The question of the number of normalized Latin squares of order $n$, $\ell(n)$, is interesting. For a given number $n$, however, no explicit expression for $\ell(n)$ is known; Up to n = 8, $\ell(n)$ has been calculated ($n = 7$ by A. Sade, and $n = 8$ by Mark Wells with the computer MANIAC 
in 8 hours of computing time [appears in: J. of Combinatorial Theory]):\blue{\footnote{\blue{To date these numbers have been computed up to $n=11$. See\\
B.\,D.~McKay and I.\,M.~Wanless, On the number of Latin squares,
{\it Ann.\ Comb.} {\bf9} (2005), 335--344.\\
D.\,S. Stones, The many formulae for the number of {L}atin rectangles, 
Electron. J. Combin. {\bf17} (2010), \#A1.
}}}
$$\begin{array}{c|cccccccc}
n & 1 & 2 & 3 & 4 & 5 & 6 & 7 & 8 \\ \hline
\ell(n) & 1 & 1 & 1 & 1 & 56 & 9\,408 & 16\,942\,080 & 535\,281\,401\,856
\end{array}$$
It is noteworthy that in these small known values for $\ell(n)$ for $n \geq 4$, quite high powers of $2$ occur.\blue{\footnote{\blue{This same observation was 
independently made by Alter. It was subsequently proved that increasing 
powers of any prime (not just 2)
divide $\ell(n)$, although we are currently unable to explain the apparent 
rapid growth in the power of 2. For details, see:\\
R.\,Alter, How many Latin squares are there?
{\it Amer.\ Math.\ Monthly}, {\bf 82} (1975) 632--634.\\
D.\,S.~Stones and I.\,M.~Wanless, 
Divisors of the number of Latin rectangles,
{\it J. Combin. Theory Ser. A\/} {\bf117} (2010), 204--215.
}}}

Note that the $\ell(n)$ appear as coefficients of a certain power series related to MacMahon's ``Master Theorem'', which, however, is not suitable for the calculation of $\ell(n)$. 

Otherwise, quite little is known about Latin squares. Here are some known results:

\begin{theorem}[M. Hall]
 Every $r \times n$ Latin rectangle ($r \leq n$) can be completed to a Latin square of order $n$.
 \end{theorem} 
 \begin{proof}
 We assume that the given rectangle as normalized. Let $S_i$ be the subset of $S$ that consists of the digits not occurring in the $i^{th}$ column ($i = 1,2, \dots, n$). From $S_1, S_2, \dots , S_n$, we define $A$ to be the incidence matrix defined by:
$$A = \left[\alpha_{ik}\right]_{n \times n} \qquad \text{where } \alpha_{ik} = \begin{cases} 1 & \text{if } k \in S_i \\ 0 & \text{otherwise.}\end{cases}$$
  
Because $|S_i| = n-r$, each row of $A$ contains exactly $n-r$ ones. By the definition of Latin rectangles, any given number occurs in exactly $r$ of the sets $S \setminus S_i~(i = 1, \dots, n)$. Thus, it occurs in exactly $n-r$ of sets $S_i$. Therefore, every column of $A$ contains exactly $n-r$ ones.
 
 However, according to the Birkhoff-K\"onig Theorem (later quoted), $A$ is the sum of $n-r$ permutation matrices $P_j$:
$$A=P_1 + \cdots + P_{n-r}.$$

If the $(r + i)^{th}$ row is now set equal to the permutation of the digits ($1, \dots,n$) given by $P_i$, i.e.,
$$\left[\alpha_{r+i,j}\right] = P_i \left[\begin{array}{c}1 \\ \vdots \\ n\end{array}\right]$$
in the sense of matrix multiplication, one obtains a Latin square. Obviously, there are no repetitions in the rows. Moreover, there are no repetitions in the columns, since the newly added digits of the column $i$ are exactly the elements of $S_i$.
 \end{proof}

{The following slightly generalized result also holds (here without a proof).}

\begin{theorem}
An $r \times s$ Latin rectangle can be completed to a Latin square of order $n$ if and only if $$N(i) \geq r + s-n \qquad \text{ for all } i=1,\dots,r,$$ where $N(i)$ is the number of times the symbol $i$ appears in the given $r \times s$ Latin rectangle.
\end{theorem}

The above result follows from Theorem 1.1 for $s = n$.

The problem addressed can be made more general in the following way: Let $S$ be the set consisting of the numbers $1,\dots, n$ and the symbol \texttt{x} (``empty cell''). Given an $n \times n$ matrix of elements of $S$, what are the necessary and sufficient conditions for this matrix to be completed to a Latin square by substitution of each \texttt{x} by digits (not necessarily the same digits)? Above, we have dealt with a special case:
$$\left[
\begin{array}{cccc}\multicolumn{4}{c}{\texttt{<Digits>}} \\ \hline
\texttt{x} & \texttt{x} & \cdots & \texttt{x} \\
\vdots &  &  & \vdots \\
\texttt{x} & \texttt{x} & \cdots & \texttt{x} \\
 \end{array}\right]$$

{Hereinafter, a line of a matrix refers to either one of its rows or columns.}

Given a square matrix $\left[a_{ik}\right]$ of order $n$, a set $P = \{(i, k)\}$ of size $n$ is a path if no two of its elements share the same row or column:
$$\begin{array}{crlc}
(i, k) \in P&\implies&& (i, l) \notin P \text{ if } l \neq k\\ && \text{and }&(j, k) \notin P \text{ if } i \neq j.
\end{array}$$

A path $P$ is called a transversal if all digits from $1$ to $n$ occur among $\left[a_{ik}\right], \text{ for } (i,k) \in P$. For example, the main diagonal of each square matrix is a path, as is the back-diagonal. 
The main diagonal of the Latin square
$$\left[\begin{array}{ccc}1&2&3\\2&3&1\\3&1&2\end{array}\right]$$
is a transversal path, but its back-diagonal is not. The Latin square
$$\left[\begin{array}{cc}1&2\\2&1\end{array}\right]$$
does not have any transversal paths. The first example seems to be a special case of a more general fact; We formulate the following conjecture:

A Latin square of odd order always has a transversal
path\blue{\footnote{\blue{This is \cjref{cj:odd}}}}. This assumption
was verified in the case of $n = 5$ (H. J. Ryser). For general $n$, it
can easily be proved under the extra assumption that it is a symmetric
Latin square.

In this case, the main diagonal is always a transversal path. It
suffices to show that it contains a $1$, but this follows from the
fact that $1$ occurs exactly $n$ times in the square, and because of
the symmetry, $\frac{n - k}{2}$ times below the principal diagonal,
where $k$ is the number of the 1s on the main diagonal; $k$ cannot be
zero because $n$ is odd.

\section{Orthogonal Systems of Latin Squares}

Let $A_1$, $A_2$ be Latin squares of order $n \geq 3$. Form the ``superimposition square'' of the ordered pairs $\left[\left(a_{ij}^{(1)}, a_{ij}^{(2)}\right)\right]$ (Euler's ``Graeco-Latin'' squares). If the $n^2$ pairs in this square are all different, then $A_1$ and $A_2$ are called orthogonal.

For example,
$$
\left[\begin{array}{ccc}1&2&3\\2&3&1\\3&1&2\end{array}\right]
\quad \text{ and } \quad
\left[\begin{array}{ccc}1&2&3\\3&1&2\\2&3&1\end{array}\right]
$$
are orthogonal Latin squares. To find the second orthogonal square, the three transversal paths of the first are occupied by the same numbers. A set of Latin squares is called an orthogonal system of Latin squares when its elements are pairwise orthogonal.

The following shows that orthogonal systems cannot be too large.

\begin{theorem}
 For any orthogonal system $A_1, \dots , A_t$ $(A=[a_{ij}^{(\tau)}], \tau=1,\dots, t)$ of $t$ Latin squares of order $n$ ($t \geq 2$, $n \geq 3$), $$t \leq n-1.$$
 
 \begin{proof}
 Obviously, permuting the elements of the matrix does not change the orthogonality of squares, so we can assume that $A_1, \dots , A_t$ have the same first row consisting of the numbers $1, \dots, n$ in the natural order. The $a_{21}^{(\tau)} (\tau = 1,\dots,t)$ must now all be different, and they cannot be 1. From this, the theorem follows.
 \end{proof}
\end{theorem}

If $t = n-1$, then we call it a 
{complete system of orthogonal Latin squares}

\begin{theorem}
 For $n = p^a \geq 3$ (for $p$ prime) there exists a complete system of orthogonal Latin squares of order $n$ (proved by Veblen and Bassey).

 \begin{proof}
   Let $S = GF(p ^ a) = \left\{a_0 = 0, a_1 = 1, a_2, \dots, a_ {n-1}\right\}$. A complete system of orthogonal squares is defined by the fact that $A_e = \left[a_ {ij}^{(e)}\right]$ with $a_ {ij}^{(e)} = a_ea_i + a_j$. First, we verify that these are Latin squares:

  \begin{itemize}
   \item If elements are in the same row, then $$a_ea_i + a_j = a_ea_i + a_{j'}$$
from which $a_j = a_ {j'}$ and thus $j = j'$;

   \item If elements are in the same column, then $$a_ea_i + a_j = a_ea_ {i'} + a_j$$
and thus, $a_i = a_{i'}$ and $i = i'$.
  \end{itemize}
  
  \textbf{Proof of Orthogonality}
  
  Assume that
  $$\left(a_{ij}^{(e)}, a_{ij}^{(f)}\right) = \left(a_{i’j’}^{(e)}, a_{i’j’}^{(f)}\right)$$
with $e \neq f$ and $(i, j) \neq (i', j')$. From this would follow that
$$a_ea_i+a_j = a_ea_{i'}+a_{j'}$$ and $$a_fa_i+a_j = a_fa_{i'}+a_{j'}.$$

Subtraction of these equations give
$$a_i(a_e-a_f) = (a_e-a_f)a_{i'}, \qquad \text{so } a_i = a_{i'},$$
and consequently also $a_j = a_j'$, which contradicts the assumption.
 \end{proof}
\end{theorem}

A difficult question is the uniqueness of complete orthogonal systems. Uniqueness could be demonstrated up to order 8; For $n = 9$, it is unknown.\blue{\footnote{\blue{It has since been proved that there are exactly 4 projective planes of order 9, and that these correspond to 19 complete sets of MOLS of order 9. See\\
P.\,J. Owens and D.\,A. Preece.
Aspects of complete sets of {$9\times 9$} pairwise orthogonal Latin squares.
{\em Discrete Math.} {\bf 167/168} (1997), 519--525.
}}}

On the other hand, we know more about the existence of orthogonal (i.e., not complete) systems:

\begin{theorem}[MacNeish]
 Let $$n=p_1^{\alpha_1}\cdots p_N^{\alpha_N} (p_i\text{ prime}), \quad t = \min_{i=1,\dots,N} \left(p_i^{\alpha_i}-1\right) \geq 2.$$ Then there are $t$ pairwise orthogonal Latin squares of order $n$.
\end{theorem}

To prove of this theorem, we need the following two lemmas:

\begin{lemma}
A system of $t$ orthogonal Latin squares $A_1, \dots, A_t$ corresponds one-to-one to the following schema

\vskip5pt

\begin{tikzpicture}
\draw (-1.5,3.4) rectangle (1.6,-3.4);

\matrix [matrix of nodes,row sep=-\pgflinewidth]
{
1 & 1 & \qquad~ & \qquad & \dots \\
1 & 2 & \quad~ & \quad & \dots \\
\vdots & \vdots & \quad & \quad & \dots \\
1 & n & \quad & \quad & \dots\\
\\
2 & 1 & \quad~ & \quad & \dots \\
2 & 2 & \quad~ & \quad & \dots \\
\vdots & \vdots & \quad & \quad & \dots \\
2 & n & \quad & \quad & \dots\\
. & . & . & . & \dots \\
n & 1 & \quad~ & \quad & \dots \\
n & 2 & \quad~ & \quad & \dots \\
\vdots & \vdots & \quad & \quad & \dots \\
n & n & \quad & \quad & \dots\\
};

\node [draw] at (-0.3,2.2) {\rotatebox{90}{{\small Row 1 of $A_1$}}};
\node [draw] at (0.4,2.2) {\rotatebox{90}{{\small Row 1 of $A_2$}}};

\node [draw] at (-0.3,0.1) {\rotatebox{90}{{\small Row 2 of $A_1$}}};
\node [draw] at (0.4,0.1) {\rotatebox{90}{{\small Row 2 of $A_2$}}};

\node [draw] at (-0.3,-2.25) {\rotatebox{90}{{\small Row $n$ of $A_1$}}};
\node [draw] at (0.4,-2.25) {\rotatebox{90}{{\small Row $n$ of $A_2$}}};

\end{tikzpicture}
\\\noindent which has the property that every $n^2 \times 2$ submatrix contains all $n^2$ pairs formed from the numbers $1, \dots, n$.

 \begin{proof}
  The $A_\tau$ ($\tau = 1, \dots, t$), as indicated in the scheme, form a system of orthogonal Latin squares: Latin because from the form of the $1^{st}$ and $2^{nd}$ column of the scheme, it follows that there are no two identical digits in a row or column of $A_\tau$; Orthogonal because the elements of $A_\tau$ and $A_ {\tau'}$ form an $n^2 \times 2$ matrix in the schema, and therefore all $(A_\tau,A_{\tau'})$ pairs are different. This superposition, combined in the manner indicated with the $n^2 \times (t + 2)$ scheme, has the property in question.
 \end{proof}
\end{lemma}

\begin{lemma}
 If there are $t$ orthogonal Latin squares of order $n$ and $t$ of order $n'$, then there are also $t$ orthogonal Latin squares of order $n \cdot n'$.
 
 \begin{proof}
   According to Lemma 2.1, we only need to consider the corresponding $n^2 \times (t + 2)$ and $n'^2 \times(t + 2)$ schema $B = (b_{ij})$ and $B'= (b'_{kl})$. From these, we form an $(n \cdot n')^2 \times (t+2)$ scheme $B^*$ by setting
$$b^*_{(i,i'),j} = \left(b_{ij},b'_{i',j}\right) \text{ where }\begin{cases}(i,i') &= (1,1), \dots , (n,n'),\\j&=1,\dots,t+2,\end{cases}$$
according to which the elements of this new scheme are now the pairs formed by the numbers $1,\dots,n$ and $1,\dots,n'$.

 The new scheme also has the property that every $(n \cdot n')^2 \times 2$ submatrix contains all $(n\cdot n')^2$ pairs from $(1,1), \dots, (n, n')$. So by Lemma 2.1, the result follows.
 \end{proof}
\end{lemma}

The theorem follows by the application of the preceding lemma and multiple applications of Lemma 2.2.

Theorem 2.3 says nothing in the case $n = 2u$ when $u$ is odd; In all other cases, it provides the existence of at least two orthogonal Latin squares.

Euler even suggested that in this case ($n \equiv 2 \mod 4$), no pair of orthogonal Latin squares exist. This assumption is trivially correct for $n = 2$ and has also been verified for the next case of $n = 6$ (Tarry, 1900). All the more surprising is the result of Bose, Shrikhande and Parker (1960), who says that these cases are the only ones in which Euler's conjecture is true:

\begin{theorem}There are pairs of orthogonal Latin squares of each order $n \neq 2$ and $n \neq 6$.\end{theorem}

We do not give the proof here again because of its complexity, but a refutation of Euler's assumption is already contained in the more elegant proof for the case $n \equiv 10 \mod 12$ (see [1]).

\section{Projective Planes}

Let $n \geq 3$. From the standpoint of combinatorics, one can introduce projective planes as follows:

\begin{definition} A projective plane of order $n$ is given by a complete system of orthogonal Latin squares of order $n$.\end{definition}

Equivalent to this is a square zero-one matrix $A$ of order $m=n^ 2 + n + 1$ with $AA^T = nI + J$, where $$I = (\delta_ {ij}) _ {m \times m} \quad \text{ and } \quad J = (1) _ {m \times m}.$$
This is because they give a complete system of orthogonal Latin squares of order $n$; According to Lemma 2.1, we can assign a $n^2 \times (n + 1)$ scheme to these rows. The rows of this scheme are $Z_1, \dots, Z_{n^2}$. 
{Moreover, let $Z_{n^2+1}, \dots, Z_{n^2+n+1}$ be additional (pairwise distinct) symbols}, and denote by $G_{ij}$ $(i = 1,\dots,n ; j = 1,\dots,n + 1)$ the set of all $Z_k$ with $i$ in the $j$th column, combined with $Z_{n^2+j}$. By assumption, all $G_{ij}$ are different.

Finally, let $G$ denote the set of symbols $Z_{n^2+1}, \dots , Z_{n^2 + n + 1}$.

Let $A$ be the $(n^2+n+1) \times (n^2+n+1)$ incidence matrix of the
$Z_k$ with respect to the $G$'s. Then it is clear that every $G$
contains exactly $n+1$ $Z$'s, and every $Z$ is contained in precisely
$n+1$ $G$'s, which follows the matrix equation for $A$. Geometrically
speaking, the $Z$'s are the points and the $G$'s lines.

Conversely, assume $A$ is given with the above-mentioned
properties. We interpret $A$ as the incidence matrix of $n^2+n+1$
``points'' $Z_j$ and the same number of ``lines'' $G_k$. For example,
for a fixed line, $G_1$, the points $P_1,\dots, P_{n+1}$ are on
it, $Q_1,\dots, Q_{n^2}$ being the points which are not on it.

As follows from the matrix equation, exactly $n+1$ lines extend
through each point. The $n$ lines through $P_j$, different
from $G_1$, may be assigned the numbers $1,\dots,n$, for every
$j$. Thus, all lines through $G_1$ have a number $\in \{1,\dots,n\}$.

Let $C = [C_{ij}] = [\text{number of } \overline{QiPj}]$. Exactly one
line goes through each pair of points, as follows directly from the
matrix equation. It is now easy to verify that $C$ is a schema of the
desired shape from Lemma 2.1. From it, $n-1$ pairs of orthogonal Latin
squares can be deduced.

The combinatorial problem that is associated with this is the
determination of for which $n$ projective planes of order $n$
exist. For $n = p^\alpha$ with $p$ prime, the existence of a
projective plane of this order follows from Theorem 2.2. The smallest
$n$ for which this problem is not known is $n = 10$. \blue{\footnote{\blue{The existence
question for $n=10$ was subsequently settled in the negative. See\\
C.\,W.\,H. Lam, 
The search for a finite projective plane of order 10,
Amer. Math. Monthly 98 (1991), 305--318.\\
C.\,W.\,H. Lam, 
L. Thiel and S. Swiercz, 
The nonexistence of finite projective planes of order 10, 
Canad. J. Math. {\bf41} (1989), 1117--1123.}}} Here, even three
pairs of orthogonal squares could not be found (in spite of Parker's
extensive computer calculations). Parker obtained by means of
unsystematic searching 84 Latin squares of order 10. Of these, 33 had
no orthogonal `partners', 29 had exactly one, 15 had two, 2 had
three, 3 had four and 2 had five orthogonal partners. It appears that
about 50\% of all randomly generated squares have orthogonal 
partners.\blue{\footnote{\blue{This comment applies to order 10 only.
The proportion 
was estimated at a touch over 60\%, in\\
B.\,D.~McKay, A.~Meynert and W.~Myrvold, Small latin squares,
quasigroups and loops, {\it J.\ Combin.\ Des.} {\bf15} (2007), 98--119.}}}

We are still giving an idea that is useful in this context. In the Latin square
$$\left[\begin{array}{ccc}1&2&3\\2&3&1\\3&1&2\end{array}\right],$$
one finds three transversal paths, which are symbolically added by the three matrices
$$
\left[\begin{array}{ccc}1&0&0\\0&1&0\\0&0&1\end{array}\right], \quad
\left[\begin{array}{ccc}0&1&0\\0&0&1\\1&0&0\end{array}\right], \quad
\left[\begin{array}{ccc}0&0&1\\1&0&0\\0&1&0\end{array}\right]
$$
which add up to $J = (1)_{3\times3}$: the transversal paths are pairwise disjoint. If the first transversal path is occupied by $1$, the second by $2$, etc., one obtains an orthogonal partner to the given square.

It holds true that any Latin square of order $n$ has an orthogonal partner if it has $n$ pairwise disjoint transversal paths.

\textbf{Example:} 
$$
\left[
\begin{array}{cccccccccc}
\fbox{5}&1&\fbox{7}&3&4&\fbox{0}&6&\fbox{2}&8&9\\
1&2&3&4&5&6&7&8&9&0\\
\fbox{7}&3&4&5&6&\fbox{2}&8&9&0&1\\
3&4&5&6&7&8&9&0&1&2\\
4&5&6&7&8&9&0&1&2&3\\
\fbox{0}&6&\fbox{2}&8&9&\fbox{5}&1&\fbox{7}&3&4\\
6&7&8&9&0&1&2&3&4&5\\
\fbox{2}&8&9&0&1&\fbox{7}&3&4&5&6\\
8&9&0&1&2&3&4&5&6&7\\
9&0&1&2&3&4&5&6&7&8\\
\end{array}
\right]
$$
is a Latin square of order $10$ (found by Parker). If you swap all 0 and 5 which are highlighted, as well as all highlighted 2 and 7, a cyclic Latin square (with even order) is obtained, which, as is easily verified, cannot possess a transversal path, and consequently has no orthogonal partner. On the other hand, the given square has 5504 transversal paths and approximately $10^6$ orthogonal partners.\blue{\footnote{\blue{Ryser is quoting Parker's estimate here, but the true figure
is 12\,265\,168. See\\
B.\,M.~Maenhaut and I.\,M.~Wanless, Atomic Latin
squares of order eleven, {\it J.\ Combin.\ Designs\/}, {\bf 12}
(2004), 12--34.
}}}

The above representation of Latin squares mean that $n$ zero-one matrices which indicate the position of the transversal paths can be further refined. To form a given Latin square $[a_{ij}]$ form the $n$ matrices $$[d_ {ijk}] \text{ with } d_{ijk} = \begin{cases}1 & \text{for }a_{ij}=k \\ 0 & \text{otherwise.}\end{cases}$$

These form a so-called Latin cube, a three-dimensional zero-one matrix, which has exactly one 1 in every (axis-parallel) line. This output square has an orthogonal partner exactly if the cube can be decomposed into other cubes with a 1 in each axis-parallel plane. (See articles by Ryser and Jurkat, published in the J. of Algebra).\blue{\footnote{\blue{In particular, see
``Extremal configurations and decomposition theorems I'',
{\it J. Algebra} {\bf8} (1968) 194--222.}}}

\section{Two combinatorial results for zero-one matrices}

On the basis of the following two problems, two characteristic conclusions of combinatorics are to be presented.

\begin{problem}
Let $A = (a_{ij})$ be a zero-one matrix of order $m \times n$ such that every element of $A^TA$ is positive and $A$ has no sub-matrix of the form
$$
\left[\begin{array}{ccc}0&1&1\\1&0&1\\1&1&0\end{array}\right]
$$
after permuting rows and columns. Show that $A$ has a row of all ones.
\end{problem}

\begin{problem} Let $A$ be a symmetric zero-one matrix of order $v$ such that
$$
\begin{array}{rl}
 & AA^T = (k-\lambda) I + \lambda J, \\
\text{where} & I = (\delta_{ij})_{v \times v}, \quad J = (1)_{v \times v} \\
\text{and} & 0 < \lambda < k  <v-1.
\end{array}
$$

Assuming that $k-\lambda$ is not a square, show that $A$ has exactly $k$ ones on the main diagonal. (This problem occurs with $(b, k, \lambda)$ configurations, see [1]).
\end{problem}

\textbf{Proof of Problem 1:}
We introduce the proof by induction over the number of columns. For $n = 1$ and $n = 2$, the assertion is true. So let $n \geq 3$.

Let $B$ denote the matrix which is formed by removing the first column from $A$. Obviously, $B^TB$ is a submatrix of $A^TA$ and thus, has only positive elements. This implies, according to the induction requirement, that B has a row of ones; For example, the first line of B.

If $a_ {11} = 1$, the theorem is proved. Therefore, we consider only the non-trivial case $a_ {11} = 0$.
Let us repeat the statements which have just been applied, omitting the second column instead of the first column of $A$, the row of all ones must be in the last $m-1$ lines. Thus, we obtain that $A$ has the following form (up to permutation of rows and columns):
$$
\left[
\begin{array}{cccccc}
0 & 1 & 1 & 1 & \dots & 0 \\
1 & 0 & 1 & 1 & \dots & 0 \\
1 & 1 & 0 & 1 & \dots & 0 \\
. & . & . & . & \dots & .
\end{array}
\right]
$$
This matrix has a submatrix of the forbidden form, and consequently, the unfavourable case could never occur.

\textbf{Proof of Problem 2:}
We examine the spectrum $\{\lambda_1, \dots , \lambda_v\}$ of $A$.

From the matrix equations, it follows that $A$ has exactly $k$ ones in each row; Hence, $\lambda_1 = k$ is an eigenvalue of $A$ (with the eigenvector $\left(\begin{array}{ccc}1&\cdots&1\end{array}\right)^T$).

Because $AA^T = A^2$, $(k-\lambda) I + \lambda J$ has eigenvalues $\lambda_1^2, \lambda_2^2, \dots, \lambda_v^2$. Now, however, $J$ has only the eigenvalue $v$ which is different from $0$, since it has rank $1$. Therefore, the remaining eigenvalues of $AA^T$ are all $k-\lambda$. I.e., the spectrum of $A$ is $= \{k, \sqrt{k-\lambda}, -\sqrt{k-\lambda}\}$.

And so applies to the trace:
$$\tr(A) = k + a \sqrt{k-\lambda} - b \sqrt{k-\lambda} \quad \text{ with integers } a, b.$$
But since $\tr(A)$ must be an integer and $k-\lambda$ is not a square,
$a = b$. The trace of A is thus $= k$.

\section{The Marriage theorem}

In this section, we prove two related propositions.

\begin{theorem}[K\"onig-Egerv\'ary] 
Let $A$ be a zero-one matrix of order $m \times n$. The minimum number
of the lines containing all of the ones in $A$ is equal to the maximum
number of ones of which no two share a line.
\end{theorem}

Example: $$
\begin{array}{cccccc}
&\downarrow & \downarrow \\
\rightarrow&1 & 0 & \underline{1} & 1 \\
&0 & \underline{1} & 0 & 0 \\
&\underline{1} & 0 & 0 & 0 \\
&1 & 0 & 0 & 0 \\
\end{array}
$$

\begin{proof}
Let $r$ be the minimum number of the lines of $A$ that contain all of
the ones of $A$, and let $e$ be the maximum number of the ones that do
not share a line.

Note that $r \geq e$, since none of the $r$ elements of a minimum set
of lines contains two ones of a maximal system of ones. We
will show that $r \leq e$.

For this purpose, we fix a minimal set of lines consisting of $z$
rows and $s$ columns ($z + s = r$). Without loss of generality, we may
assume that they are the first $z$ rows and $s$ columns of the matrix:
$$A = \left[ \begin{array}{cc}A_{11} & A_{12} \\ A_{21} & 0 \end{array} \right]$$
where $A_ {11}$ is a $z \times s$ matrix. First, let $0 < z < m$.

Now $A_{12}$ has a maximal system of $z$ ones, since otherwise --- we may assume that the theorem is already proven for $A_{12}$ --- $A_{12}$ would have a minimal system of size less than $z$, which in combination with the first $s$ columns of $A$ would yield a system of less than $r$ lines that cover all ones of $A$, contradicting the minimality of $r$. Similarly,  $A_{21}$ has a maximal system of $s$ ones. Thus, in $A$, we have a maximal system of at least $s+z=r$ ones, so $e\geq r$. In the cases $z = m$ and $z = 0$, one can consider $A$ and/or $A^T$ as the incidence matrix of a system of $n$ or $m$ sets for which the assertion is a consequence of the following theorem. 
\end{proof}

Now let's discuss the marriage theorem of P. Hall: First, a definition.

\begin{definition}
Let $S_1, \dots, S_m$ be subsets of $S$. An $m$-tuple $(a_1, \dots, a_m)$ of distinct elements of $S$ with $a_i \in S_i$ ($i = 1, \dots, m$) is called an individual representative system.
\end{definition}

The marriage theorem provides a necessary and sufficient condition for
a system $S_1, \dots ,S_m$ of subsets to have an individual
representative system.

\begin{theorem}
$S_1, \dots ,S_m$ have an individual representative system if and only if $|S_{i_1} \cup \dots \cup S_{i_k}| \geq k$ for all $k = 1, \dots, m$ for all $k$-subsets $\{i_1, \dots, i_k\} \subseteq \{1, \dots, m\}$.
\end{theorem}

\begin{proof}
It is only necessary to prove the sufficiency of the condition. For $m = 1$ it is trivial. To prove it for $m$, if it is true for all $m' < m$, we break into two cases.\\
\textbf{Case 1}: For all $k = 1, \dots, m-1$ and all k-subsets $\{i_1, \dots, i_k\} \subseteq \{1, \dots, m\}$, $|S_{i_1} \cup \dots \cup S_{i_k}| \geq k+1$.\\
\textbf{Case 2}: For some $k \in \{1, \dots, m-1\}$, there exists a $k$-subset $\{i_1, \dots, i_k\} \subseteq \{1, \dots, m\}$, $|S_{i_1} \cup \dots \cup S_{i_k}| = k$.

In the first case, fix an $a_1 \in S_1$ and set $S_j' = S_{j} - \{a_1\}$ ($j = 2,\dots,m$). This set system then satisfies the condition of the theorem and thus has, by induction, an individual representative system for $S_1,\dots,S_m$.

In the second case, the exceptional indices are, without loss of generality, $\{i_1,\dots,i_k\} = \{1,\dots,k\}$. By induction, $S_1, \dots, S_k$ have an individual representative system $(a_1,\dots, a_k)$. To the sets $S_j^* : = S_j - \{a_1, \dots, a_k\}$ ($j = k + 1, \dots, m$), we may apply the induction hypothesis to find an individual representative system $(a_{k+1}, \dots,a_m)$ since when
$$|S_{k+1}^* \cup \cdots \cup S^*_{k+k^*}| < k^*$$
then
$$|S_1 \cup \cdots \cup S_k \cup S_{k+1}^* \cup \cdots \cup S_{k+k^*}| < k + k^*$$
a contradiction.
Thus, $(a_1,\dots,a_m)$ is an individual representative system of $S_1,\dots,S_m$.
\end{proof}

This completes the proof of Theorem 5.2; A sharper statement is proved in [1].

\section{Permanents}

Let $A = [a_{ij}]$ be an $m \times n$ matrix ($m \leq n$). The permanent of $A$ is defined by
$$\per(A) = \sum A_{1i_i} \cdots A_{mi_m},$$
where $(i_1,\dots,i_m)$ runs over all permutations of $\{1,\dots,m\}$.

The relationship with the individual representative systems is the
following: Let $S_1,\dots,S_m$ be subsets of an $n$-element set $S$
with $m \leq n$, and let $A$ be the incidence matrix of the
system. The number of all individual representative systems is
$\per(A)$.

The permanent of a matrix is often very difficult to calculate. The
following formula, which we cite without proof from [1], reduces the
computation a bit with $n \times n$ matrices:
\begin{equation}
\per (A) = S (A) -\sum_{A_1} S(A_1) + \sum_{A_2} S(A_2) - \cdots + (-1)^{n-1} \sum_{A_{n-1}} S(A_{n-1})
\end{equation}
where $A_r$ is an $n \times (n-r)$ submatrix of $A$ and $S(A_r)$ is the product over all row sums of $A_r$.

For example, if we compute the permanent of $J-I$ according to this formula, we get the following expression:
$$\per(J-I) = \sum_{r=0}^{n-1} (-1)^r \binom{n}{r} (n-r)^r(n-r-1)^r.$$

On the other hand, of course, $\per(J-I)$ is equal to the number of derangements $D_n$ (`derangement number') of $n$ objects, thus $= n!\left(1-\frac{1}{1!} + \cdots + (-1)^n\frac{1}{n!}\right)$, so
$$n!\left(1-\frac{1}{1!} + \cdots + (-1)^n\frac{1}{n!}\right) = \sum_{r=0}^{n-1} (-1)^r \binom{n}{r} (n-r)^r(n-r-1)^r.$$
$$X=
\left[
\begin{array}{ccccc}
x & y & & &  \\
 & x & y & & 0\\
 & & \ddots & \ddots \\
 &0 & & x & y \\
 y & & & & x
\end{array}
\right]
$$

We will compute $\per(X)$ in two different ways.

On the one hand, $\per(X) = x^n + y^n$. On the other hand, by applying (6.1),
$$\begin{array}{rcl}
\per(A) & = & \displaystyle\sum_{r=0}^{n-1} (-1)^rS(X_r) \\
& = & \displaystyle\sum_{r=0}^{n-1} (-1)^ra_r(xy)^r(x+y)^{n-2r},
\end{array}$$
where $a_r$ is the number of possibilities to select $r$ non-neighbouring elements from a circle of $n$ objects. According to a lemma from Kaplanski (see [1] p. 34):
$$a_r = \frac{n}{n-r}\binom{n-r}{r} \quad \text{(obviously }a_r= 0 \text{ for } r > \left\lfloor\frac{n}{2}\right\rfloor).$$

Similar to the case of determinants, the question arises for which
matrices $A$ is $\per(A) \neq 0$.

A square matrix of order $n$ is called doubly stochastic if its
elements are nonnegative and all row and column sums are equal to 1.

For doubly stochastic matrices $A$, $\per(A)$ is positive. This follows immediately from the Birkhoff-K\"onig Theorem:

\begin{theorem}
Let $A$ be doubly stochastic of order $n$. Then $A = c_1P_1 + \cdots + c_tP_t$, where each $P_\tau$ is a permutation matrix and $c_1,\dots, c_t$ are nonnegative real numbers with $\sum c_i = 1$.
\end{theorem}
\begin{proof}
$A$ has $n$ positive elements which occur in different rows and columns; Otherwise, Theorem 5.1 applied with the positive elements of A with $z$ rows and $s$ columns where $z + s < n$, we have that $A$ would not be doubly stochastic. Now let $P_1$ be the permutation matrix, which has ones at the corresponding entries; Let $c_1$ be the smallest of the selected $n$ positive elements. Then $A - c_1P_1$ is a matrix of nonnegative elements whose row and column sums are all equal to $1-c_1$.
Repeated application of the preceding conclusion to $A-c_1P_1$ instead of to $A$ and so on finally yields the desired representation.
\end{proof}

\textbf{Note}: If $A$ is a zero-one matrix, with identical row and column sums, then one finds in this way, of course, a representation with $c_1 = 1$ for all $i$.

A sharp lower estimate of $\per(A)$ for doubly stochastic matrices $A$
is not yet known. The conjecture of van der Waerden, which states
that $$\per(A) \geq \frac{n!}{n^n}$$ holds, with equality only if
$A=\frac1nJ$, has only been verified for small values of $n$ ($n \leq
4$).\blue{\footnote{\blue{The van der Waerden conjecture has since been
proved. See\\ 
J.\,H. van Lint, 
The van der Waerden conjecture: two proofs in one year, 
Math. Intelligencer {\bf4} (1982), 72--77.\\
G.\,P.~Egorychev, 
Solution of the van der Waerden problem for permanents, 
{\em Soviet Math.\ Dokl\/}, {\bf 23} (1981), 619--622.\\
D.\,I.~Falikman, 
Proof of the van der Waerden conjecture regarding the permanent of a 
doubly stochastic matrix, {\em Math.\ Notes\/} {\bf 29} (1981), 475--479.
}}}

We briefly check for $n = 2$: Obviously a doubly stochastic matrix has the form
$$\left[ \begin{array}{cc}x & 1-x \\ 1-x & x \end{array} \right]$$
with $0 \leq x \leq 1$, and therefore
$$\per \left[ \begin{array}{cc}x & 1-x \\ 1-x & x \end{array} \right] = x^2 + (1-x)^2 \geq \frac12 = 2!/2^2.$$

Marcus and Minc (in: Permanents, Am. Math. Monthly, 72, (1965), 577-591) were able to prove the weaker result:

\begin{theorem}
For a doubly stochastic matrix $A = [a_ij]$, there exists a
permutation $\sigma$ of $\{1,\dots,n\}$ such that 
$$\prod_{j=1}^{n} a_{j\sigma(j)} \geq \frac{1}{n^n}.$$
\end{theorem}

One may wish to generalize van der Waerden's conjecture by the
following: $$\per(AB) \leq \min(\per(A),\per(B)).$$

This inequality was, however, rejected by Jurkat by the counter-example
$$A = \frac{1}{24} \left[ \begin{array}{ccc} 11 & 5 & 8 \\ 13 & 11 & 0 \\ 0 & 8 & 16 \end{array} \right]$$
$$B = \frac{1}{2} \left[ \begin{array}{ccc} 1 & 1 & 0 \\ 1 & 1 & 0 \\ 0 & 0 & 2 \end{array} \right]$$
In this case, $\per(A) = \frac{3808}{13824} < \per(AB) = \frac{3840}{13824}$. A further conjecture has already been expressed, namely by $\per(AA^T) < \per(A)$. But it too could be disproved, by M. Newman with the counter-example
$$A = \frac{1}{2} \left[ \begin{array}{cccc} 1 & 1 & 0 & 0 \\ 0 & 1 & 1 & 0 \\ 0 & 0 & 1 & 1 \\ 1 & 0 & 0 & 1 \end{array} \right]$$
for which $\per(AA^T) = \frac{9}{64} > \per(A) = \frac{8}{64}$.

The van der Waerden conjecture also makes sense for matrices other
than doubly stochastic. Consider, e.g., the class $U(R, S)$ of all 
$m\times n$ matrices of nonnegative elements with given vectors 
$R =(r_i)$ and $S = (s_k)$ for the row and column sums. Substituting 
$R = S = (1 .... 1)_n$, we again have the class of doubly stochastic
matrices; Here at least the maximum of the permanent function is
known; it is equal to 1 and is attained precisely for the permutation
matrices.

Now let $m = n$ and $R = (1,n+1,\dots,n+1)$, $S = (n+1,\dots,n+1,1)$. The class $U(R, S)$ contains, for example,
$$
\left[
\begin{array}{cccc}
1 &  &  & 0 \\
n & 1 &  &  \\
 & \ddots & \ddots \\
0 &  & n & 1
\end{array}
\right]
$$

The permanent of such matrices is always positive, and it is useful to ask for the maximum and minimum of the permanent function. Is the minimum of $\per(A)$ over this class $= 1$ if only integer elements are allowed? A sharp estimate for the upper bound was shown recently (see [8]):

\begin{theorem}
Let $A \in U(R, S)$, where $R = (r_j)$, $S = (s_j)$, and $r_1 \leq r_2 \leq \cdots \leq r_n$, $s_1 \leq s_2 \leq \cdots \leq s_n$. Then $$\per(A) \leq \prod_{j=1}^{n} \min(r_j,s_j).$$
\end{theorem}

In view of the difficulties of these problems, it is useful first to restrict to the class $R$ of all zero-one matrices of order $n$ with row and column sums $= k$ ($1 \leq k \leq n$) and determine the extremes of the permanent over $R$.

For this, there is initially a lower estimate of M. Hall: For $A \in R$, 
$\per(A) > k!$. 
And for the more general case of row sums $r_1, r_2, \dots, r_n$, 
an inequality of Minc
$$\per(A) \leq \prod_{j=1}^{n} \frac{r_j+1}{2}.$$
In addition, Minc conjectured\blue{\footnote{\blue{The first part of this conjecture has also been proved. See\\
L.\,M.~Br\`egman, Some properties of nonnegative matrices
and their permanents, {\it Soviet Math.\ Dokl.\/} {\bf14} (1973), 945--949.\\
A. Schrijver, A short proof of Minc's conjecture, 
{\it J. Combinatorial Theory Ser. A} {\bf25} (1978), 80--83. 
}}}
$$\per(A) \leq \prod_{j=1}^n (r_j!)^{1/r_j}$$
and
$$\per(A) \leq \prod_{j=1}^n (r_j!)^{1/n} \cdot \left(\frac{r_j+1}{2}\right)^{\frac{n-r_j}{n}}.$$

Combining van der Waerden's conjecture of the first bound of Minc's, one obtains the estimate
$$\left(\frac{n!}{n^n}\right)^r \cdot \prod_{j=1}^r (n+1-j)^n \leq L(r,n) \leq \prod_{j=1}^r (n+1-j)^{\frac{n}{n+1-j}},$$
where $L (r, n)$ is the number of (non-normalized) $r \times n$ Latin rectangles; The first inequality is based on van der Waerden's conjecture and the second of Minc's conjectures.


\noindent
Basic facts dealt with in this paper can be found in detail in the literature:

\medskip

\begin{enumerate}

\item[{[1]}] Ryser, H. J., Combinatorial Mathematics, 
Carus Mathematical Monographs {\bf14}, 1963.

\item[{[2]}] Hall, M., Combinatorial Theory, Ginn \& Co., 1967.

\end{enumerate}

\medskip

\noindent
Many things that could only be hinted at here without proof are
contained in recent journal articles by H.J. Ryser and co-authors:

\smallskip

\begin{enumerate}

\item[{[3]}] Ryser, H. J., Maximal determinants in combinatorial investigations,
Can. J. Math. 8, 245--249 (1956).

\item[{[4]}] Ryser, H. J., Inequalities of compound and induced matrices with
applications to combinatorial analysis, 
Illinois J. Math. 2, 240--253 (1958).

\item[{[5]}] Fulkerson, D. R., and Ryser, H. J.,
Multiplicities and Minimal widths for $(0,1)$-matrices,
Can. J. Math. 14, 498--508 (1962).

\item[{[6]}] Fulkerson, D. R., and Ryser, H. J.,
Width sequences for special classes of $(0,1)$-matrices,
Can. J. Math. 15, 371--396.

\item[{[7]}] Jurkat, W. B., and Ryser, H. J.,
Matrix factorizations of determinants and permanents,
J. Algebra 3, 1--27.

\item[{[8]}] Jurkat, W. B., 
Term ranks and permanents of nonnegative matrices,
J. Algebra 5, 342--357 (1967).

\end{enumerate}


\begin{thebibliography}{99}

\bibitem{ABW16}
R. Aharoni, J.~Bar\'at and I.\,M.~Wanless,
Multipartite hypergraphs achieving equality in Ryser's conjecture,
{\it Graphs Combin.\/} {\bf32} (2016), 1--15.

\bibitem{AA04}
S.~Akbari and A.~Alireza, 
Transversals and multicolored matchings, 
{\it J.\ Combin.\ Des.} {\bf12} (2004), {325--332}.

\bibitem{Bal90}
K.~Balasubramanian, 
On transversals in latin squares, 
{\it Linear Algebra Appl.} {\bf131} (1990), {125--129}.

\bibitem{CW05}
P.\,J.~Cameron and I.\,M.~Wanless, Covering radius
for sets of permutations, {\it Discrete Math.\/} {\bf293} (2005) 91--109.

\bibitem{FHMW17}
N.~Franceti\'c, S.~Herke, B.\,D.~McKay and I.\,M.~Wanless, 
On Ryser's conjecture for linear intersecting multipartite hypergraphs,
{\it European J. Combin.} {\bf61} (2017), 91--105.

\bibitem{Hen71}
J.\,R.~Henderson,
Permutation Decompositions of $(0,1)$-matrices and decomposition
transversals, Ph.D.\ Thesis, Caltech (1971),
\texttt{http://thesis.library.caltech.edu/5726}.

\bibitem{MMW06}
B.\,D.~McKay, J.\,C.~McLeod and I.\,M.~Wanless, 
The number of transversals in a Latin square, 
{\it Des.\ Codes Cryptogr.}, {\bf40} (2006), 269--284.

\bibitem{plexes} 
I.\,M.~Wanless, A generalisation of transversals for latin squares, 
{\it Electron.\ J.\ Combin.} {\bf9} (2002), R12.


\bibitem{surv1}
I.\,M.~Wanless, Transversals in latin squares, 
{\it Quasigroups Related Systems} {\bf15}, (2007) 169--190.


\bibitem{surv2}
I.\,M.~Wanless,
``Transversals in Latin squares: A survey'', in
R.~Chapman (ed.), {\it Surveys in Combinatorics 2011},
London Math. Soc. Lecture Note Series {\bf392}, 
Cambridge University Press, 2011, pp403--437.


\end{thebibliography}

\begin{thebibliography}{}

\end{thebibliography}
\end{document}